\documentclass[12pt]{amsart}
\usepackage{xcolor}
\usepackage{amsmath,amsthm,amssymb,enumitem,verbatim,stmaryrd,xcolor,microtype,graphicx,aliascnt,mathtools}
\usepackage[T1]{fontenc}
\usepackage[utf8]{inputenc}
\usepackage[english]{babel} 
\usepackage[top=3.5cm,bottom=3.5cm,left=3.5cm,right=3.5cm]{geometry}
\usepackage[bookmarksdepth=2,linktoc=page,colorlinks,linkcolor={red!70!black},citecolor={red!80!black},urlcolor={blue!80!black}]{hyperref}
\usepackage{tikz}\usetikzlibrary{matrix,arrows,decorations.markings,shapes}
\usepackage{tikz-cd}
\usepackage[mathcal]{euscript}                               

\usepackage{mathptmx}                                        
\usepackage[colorinlistoftodos,bordercolor=orange,backgroundcolor=orange!20,linecolor=orange,textsize=footnotesize]{todonotes} \makeatletter \providecommand \@dotsep{5} \def\listtodoname{List of Todos} \def\listoftodos{\@starttoc{tdo}\listtodoname} \makeatother 

\allowdisplaybreaks 

\usepackage{etoolbox}
\makeatletter
\patchcmd{\@startsection}{\@afterindenttrue}{\@afterindentfalse}{}{}             
\patchcmd{\part}{\bfseries}{\bfseries\LARGE}{}{}
\patchcmd{\section}{\scshape}{\bfseries}{}{}\renewcommand{\@secnumfont}{\bfseries} 
\patchcmd{\@settitle}{\uppercasenonmath\@title}{\large}{}{}
\patchcmd{\@setauthors}{\MakeUppercase}{}{}{}
\addto{\captionsenglish}{} 
\addto{\captionsenglish}{} 
\addto{\captionsenglish}{} 
\makeatother

\usepackage{fancyhdr}
\usepackage{cleveref}

\addto\extrasenglish{}  
\theoremstyle{plain}
\newtheorem{thm}{Theorem}[section] 
\newaliascnt{lemma}{thm}\newtheorem{lemma}[lemma]{Lemma}\aliascntresetthe{lemma}
\newaliascnt{cor}{thm}\newtheorem{cor}[cor]{Corollary}\aliascntresetthe{cor}
\newaliascnt{prop}{thm}\newtheorem{prop}[prop]{Proposition}\aliascntresetthe{prop}

\newtheorem*{thm*}{Theorem}
\newtheorem*{lem*}{Lemma}
\newtheorem*{cor*}{Corollary}

\theoremstyle{definition}
\newaliascnt{df}{thm}\newtheorem{df}[df]{Definition}\aliascntresetthe{df}
\newaliascnt{rem}{thm}\newtheorem{rem}[rem]{Remark}\aliascntresetthe{rem}
\newaliascnt{ex}{thm}\newtheorem{ex}[ex]{Example}\aliascntresetthe{ex}
\newtheorem{question}[thm]{Question}

\newtheorem*{df*}{Definition}
\newtheorem*{ex*}{Example}
\newtheorem*{rem*}{Remark}

\theoremstyle{remark}

\DeclareRobustCommand{\gobblefour}[5]{}    

\DeclareFontFamily{OT1}{pzc}{}                                
\DeclareFontShape{OT1}{pzc}{m}{it}{<-> s * [1.10] pzcmi7t}{}
\DeclareMathAlphabet{\mathpzc}{OT1}{pzc}{m}{it}
\DeclareSymbolFont{sfoperators}{OT1}{bch}{m}{n} \DeclareSymbolFontAlphabet{\mathsf}{sfoperators} \makeatletter\def\operator@font{\mathgroup\symsfoperators}\makeatother 
\DeclareSymbolFont{cmletters}{OML}{cmm}{m}{it}              
\DeclareSymbolFont{cmsymbols}{OMS}{cmsy}{m}{n}
\DeclareSymbolFont{cmlargesymbols}{OMX}{cmex}{m}{n}
\DeclareMathSymbol{\myjmath}{\mathord}{cmletters}{"7C}     \let\jmath\myjmath 
\DeclareMathSymbol{\myamalg}{\mathbin}{cmsymbols}{"71}     
\DeclareMathSymbol{\mycoprod}{\mathop}{cmlargesymbols}{"60}
\DeclareMathSymbol{\myalpha}{\mathord}{cmletters}{"0B}     \let\alpha\myalpha 
\DeclareMathSymbol{\mybeta}{\mathord}{cmletters}{"0C}      \let\beta\mybeta
\DeclareMathSymbol{\mygamma}{\mathord}{cmletters}{"0D}     \let\gamma\mygamma
\DeclareMathSymbol{\mydelta}{\mathord}{cmletters}{"0E}     \let\delta\mydelta
\DeclareMathSymbol{\myepsilon}{\mathord}{cmletters}{"0F}   \let\epsilon\myepsilon
\DeclareMathSymbol{\myzeta}{\mathord}{cmletters}{"10}      \let\zeta\myzeta
\DeclareMathSymbol{\myeta}{\mathord}{cmletters}{"11}       \let\eta\myeta
\DeclareMathSymbol{\mytheta}{\mathord}{cmletters}{"12}     \let\theta\mytheta
\DeclareMathSymbol{\myiota}{\mathord}{cmletters}{"13}      \let\iota\myiota
\DeclareMathSymbol{\mykappa}{\mathord}{cmletters}{"14}     \let\kappa\mykappa
\DeclareMathSymbol{\mylambda}{\mathord}{cmletters}{"15}    \let\lambda\mylambda
\DeclareMathSymbol{\mymu}{\mathord}{cmletters}{"16}        \let\mu\mymu
\DeclareMathSymbol{\mynu}{\mathord}{cmletters}{"17}        \let\nu\mynu
\DeclareMathSymbol{\myxi}{\mathord}{cmletters}{"18}        \let\xi\myxi
\DeclareMathSymbol{\mypi}{\mathord}{cmletters}{"19}        \let\pi\mypi
\DeclareMathSymbol{\myrho}{\mathord}{cmletters}{"1A}       \let\rho\myrho
\DeclareMathSymbol{\mysigma}{\mathord}{cmletters}{"1B}     \let\sigma\mysigma
\DeclareMathSymbol{\mytau}{\mathord}{cmletters}{"1C}       \let\tau\mytau
\DeclareMathSymbol{\myupsilon}{\mathord}{cmletters}{"1D}   \let\upsilon\myupsilon
\DeclareMathSymbol{\myphi}{\mathord}{cmletters}{"1E}       \let\phi\myphi
\DeclareMathSymbol{\mychi}{\mathord}{cmletters}{"1F}       \let\chi\mychi
\DeclareMathSymbol{\mypsi}{\mathord}{cmletters}{"20}       \let\psi\mypsi
\DeclareMathSymbol{\myomega}{\mathord}{cmletters}{"21}     \let\omega\myomega
\DeclareMathSymbol{\myvarepsilon}{\mathord}{cmletters}{"22}\let\varepsilon\myvarepsilon
\DeclareMathSymbol{\myvartheta}{\mathord}{cmletters}{"23}  \let\vartheta\myvartheta
\DeclareMathSymbol{\myvarpi}{\mathord}{cmletters}{"24}     \let\varpi\myvarpi
\DeclareMathSymbol{\myvarrho}{\mathord}{cmletters}{"25}    \let\varrho\myvarrho
\DeclareMathSymbol{\myvarsigma}{\mathord}{cmletters}{"26}  \let\varsigma\myvarsigma
\DeclareMathSymbol{\myvarphi}{\mathord}{cmletters}{"27}    \let\varphi\myvarphi

\DeclareMathOperator{\sign}{sign}
\DeclareMathOperator{\supp}{supp}
\DeclareMathOperator{\Minsupp}{Minsupp}

\DeclareMathOperator{\row}{Row}
\DeclareMathOperator{\col}{col}
\DeclareMathOperator{\mat}{mat}
\DeclareMathOperator{\tmat}{tmat}
\DeclareMathOperator{\CapVec}{Vec}
\DeclareMathOperator{\Cov}{Cov}
\DeclareMathOperator{\psc}{psc}

\newcommand\C{{\mathbb C}}

\newcommand\F{{\mathbb F}}

\renewcommand\H{{\mathbb H}}

\newcommand\K{{\mathbb K}}

\newcommand\N{{\mathbb N}}
\newcommand\NN{{\mathbb N}}

\renewcommand\P{{\mathbb P}}
\newcommand\Q{{\mathbb Q}}
\newcommand\R{{\mathbb R}}
\renewcommand\S{{\mathbb S}}
\newcommand\T{{\mathbb T}}

\newcommand\V{{\mathbb V}}

\newcommand\Z{{\mathbb Z}}

\newcommand\cC{{\mathcal C}}

\newcommand\cH{{\mathcal H}}

\newcommand\cX{{\mathcal X}}

\renewcommand\geq{\geqslant}
\renewcommand\leq{\leqslant}

\renewcommand\emptyset\varnothing

\newif\ifnscomms
\nscommstrue 

\newcommand{\set}[1]{\left\{ {#1}\right\}}

\title{On some notions of rank for matrices over tracts}

\author{Matthew Baker}
\address{\rm School of Mathematics, Georgia Institute of Technology, Atlanta, USA}
\email{mbaker@math.gatech.edu}

\author{Noah Solomon}
\address{\rm School of Mathematics, Georgia Institute of Technology, Atlanta, USA}
\email{noah.solomon@math.gatech.edu}

\author{Tianyi Zhang}
\address{\rm School of Mathematics, Georgia Institute of Technology, Atlanta, USA}
\email{kafuka@gatech.edu}

\thanks{The first author was supported by a Simons Foundation Collaboration Grant and by NSF research grant DMS-2154224. The second author was supported by the Department of Education's Graduate Assistance in Areas of National Need Award \#P200A240169. The authors thank Oliver Lorscheid for helpful discussions, and the anonymous referees for their unusually detailed and constructive comments on a preliminary draft of this manuscript.}

\begin{document}

\begin{abstract}
Given a tract $F$ in the sense of Baker and Bowler and a matrix $A$ with entries in $F$, we define several notions of rank for $A$.
In this way, we are able to unify and find conceptually satisfying proofs for various results about ranks of matrices that one finds scattered throughout the literature.
\end{abstract}

\maketitle


\section{Introduction}
\label{introduction}

In \cite{Baker-Bowler19}, the first author and Nathan Bowler introduced a new class of algebraic objects called {\bf tracts} which generalize not only fields but also partial fields and hyperfields.
Given a tract $F$, Baker and Bowler also define a notion of {\bf $F$-matroid}.\footnote{In fact, one finds two notions (weak and strong) of {\bf $F$-matroids} in \cite{Baker-Bowler19}; we will work exclusively in this paper with strong $F$-matroids. Over many tracts of interest, the two notions coincide.}
If $F=K$ is a field, then a $K$-matroid of rank $r$ on the finite set $E = \{ 1,\ldots,n \}$ is just an $r$-dimensional subspace of $K^n$, and matroids over the Krasner hyperfield $\K$ are just matroids in the usual sense.

\medskip

In this paper, we use the theory of $F$-matroids to define a new notion of rank for matrices with entries in a tract $F$. When $F=K$ is a field, this gives the usual notion of rank, and when $F=\K$ is the Krasner hyperfield we recover an intriguing notion of rank for matrices of zero-non-zero patterns recently introduced by Deaett \cite{Deaett}.
We also introduce a relative notion of rank for matrices over $F$ which depends on the choice of a tract homomorphism $\varphi : F' \to F$; this is, argubaly, the more important notion.
We compare these new notions of rank to more familiar notions like row rank, column rank, and determinantal rank (all of which have straightforward generalizations to matroids over tracts), providing a number of general inequalities as well as some inequalities which only hold under additional hypotheses.

\medskip

Our motivation for studying these concepts comes from a desire to unify, and find conceptually satisfying proofs for, various results about ranks of matrices that one finds scattered throughout the literature, each of which admit an interpretation as a statement about matrices with entries in a specified tract.
For example, we will give a unified proof and conceptual generalization of the following results: 

\begin{thm} \label{theorem:thm1}
\begin{enumerate}
    \item (Berman et. al.){\cite[Proposition 2.5]{Berman2008AnUB}} Let $A = (a_{ij})$ be a zero-non-zero pattern, i.e., an $m \times n$ matrix whose entries are each $0$ or $\star$, and suppose that each row of $A$ has at least $k$ $\star$s. 
    Then for any infinite field $K$, there exists an $m \times n$ matrix $A' = (a'_{ij})$ over $K$ with zero-non-zero pattern $A$\footnote{This means that $A'_{ij}=0$ (resp. $A'_{ij}\neq 0$) iff $A_{ij}=0$ (resp $A_{ij}=\star$).} and having rank at most $n-k+1$.
    \item (Alon-Spencer){\cite[Lemma 13.3.3]{alon04}} Let $A = (a_{ij})$ be a full sign pattern, i.e., an $m \times n$ matrix whose entries are each $+$ or $-$, and suppose that there are at most $k$ sign changes in each row of $A$.
    Then there exists an $m \times n$ matrix $A' = (a'_{ij})$ over $\R$ with sign pattern $A$\footnote{This means that $A'$ has non-zero entries and $A'_{ij}=+$ (resp. $A'_{ij}=-$) iff $A_{ij}>0$ (resp. $A_{ij}<0$).} and having rank at most $k+1$.
    \end{enumerate}
\end{thm}

Our proof naturally yields an extension of (2) to the non-full case, see Theorem~\ref{thm:signPatternApplication} below.

\medskip

Our results also give a unified way of viewing (and proving) results like the following.
For the statement of (2), we define a sequence $z_1,\ldots,z_n$ of complex numbers to be {\bf colopsided} if $0$ is not in their convex hull (when viewed as elements of $\R^2 \cong \C$).

\begin{thm} \label{theorem:thm2}
\begin{enumerate}
    \item (Camion-Hoffman) \cite[Theorem 3]{Camion1966}, \cite[Theorem 4.6]{Goucha2021phaseless} Let $A = (a_{ij})$ be an $n\times n$ matrix with non-negative real entries. Then every complex matrix $A'=(a'_{ij})$ with $|a'_{ij}|=a_{ij}$ for all $i,j$ is non-singular iff there exists an $n\times n$ permutation matrix $P$ and an $n\times n$ diagonal matrix $D$ with non-negative real entries such that $PAD$ is strictly diagonally dominant.
    \item (McDonald et. al.) \cite[Lemma 3.2]{Goucha2021phase}, \cite[Lemma 3.2]{Mcdonald2000}
    Let $A = (a_{ij})$ be an $n\times n$ matrix with entries in $\S^1 \cup \{ 0 \}$, where $\S^1$ is the complex unit circle. Then every complex matrix $A'=(a'_{ij})$ with ${\rm phase}(a'_{ij})=a_{ij}$ for all $i,j$ is non-singular iff there does not exist a scaling of its rows by elements of 
    $\S^1 \cup \{ 0 \}$, not all zero, such that no column is colopsided.
\end{enumerate}
\end{thm}

Theorem~\ref{theorem:thm1}(1) (resp. (2)) is obtained by applying Theorem~\ref{thm:epic} below to the natural homomorphism $K \to \K$ (resp. the natural homomorphism ${\rm sign} : \R \to \S$, where $\S$ is the sign hyperfield).

Theorem~\ref{theorem:thm2}(1) (resp. (2)) is obtained by applying Theorem~\ref{thm:quotienthyperfieldrank} below to the natural homomorphism $\C \to \V$, where $\V$ is Viro's triangle hyperfield (resp. the natural homomorphism $\C \to \P$, where $\P$ is the phase hyperfield).

\medskip

We conclude the paper with some open questions for future study.

\section{Review of tracts, hyperfields, and matroids over tracts}

Before we may define our various of notions of rank, we recall for the reader the basic ideas of tracts. For more details concerning the definitions and concepts in this section, as well as numerous examples, see \cite{Baker-Bowler19}. 

\subsection{Tracts}

Given an abelian group $G$, let $\NN[G]$ denote the group semiring associated to $G$.

\begin{df}
A {\bf tract} is a multiplicatively written commutative monoid $F$ with an absorbing element $0$ such that $F^\times := F \setminus \{ 0 \}$ is a group, together with a subset $N_F$ of $\NN[F^\times]$ satisfying:
\begin{itemize}
\item[(T1)] The zero element of $\NN[F^\times]$ belongs to $N_F$.
\item[(T2)] There is a unique element $\epsilon \neq 0$ of $F^\times$ with $1 + \epsilon \in N_F$.
\item[(T3)] $N_F$ is closed under the natural action of $F^\times$ on $\NN[F^\times]$.
\end{itemize}
\end{df}  

We call $N_F$ the {\bf  null set} of $F$, and write $-1$ instead of $\epsilon$.

Intuitively, a tract is an object generalizing the notion of a field. Instead of an operation defining addition, one has the data of the null set, which corresponds to the collection of subsets ``summing'' to zero.

\begin{df}
A {\bf homomorphism} of tracts is a map $\varphi : F' \to F$ such that $\varphi(0)=0$, $\varphi$ induces a group homomorphism from $(F')^\times$ to $F^\times$, and $\varphi(N_{F'}) \subseteq N_{F}$.
\end{df}

\begin{ex}\label{ex:Tropical}
    One important tract is the \textbf{tropical hyperfield} $\T$, which has underlying multiplicative monoid $(\R_{\geq 0},\cdot)$ with $\epsilon = 1$, and whose null set $N_\T$ consists of formal sums $\sum a_i$ where the maximum of the $a_i$ is achieved at least twice. 

    \medskip
    Given a field $K$, any non-archimedean absolute value $| \cdot | : K \to \R_{\geq 0}$
    is a tract homomorphism into $\T$. A well studied example of this is $|f| = e^{-v(f)}$, where $v$ is the valuation taking a Puiseux series in $\C\{\!\{T\}\!\}$ to the exponent of its initial term, i.e. the lowest exponent appearing with non-zero coefficient in the series expansion.
\end{ex}

\begin{df}
Given a tract $F$, a natural number $n$, and vectors $X = (X_i),Y=(Y_i) \in F^n$, we say that $X$ is {\bf orthogonal} to $Y$, denoted $X \perp Y$, if $\sum_i X_i Y_i \in N_F$.
\end{df}

\begin{df}
Given a tract $F$, we say that vectors $X_1,\ldots,X_k \in F^n$ are {\bf linearly dependent over $F$} if there exist $c_1,\ldots,c_k \in F$, not all zero, such that $\sum c_i X_i \in (N_F)^n$, and {\bf linearly independent} otherwise.
\end{df}

\subsection{Quotient hyperfields and partial fields}
There are two main types of tracts of interest to us,
quotient hyperfields and partial fields. We present definitions of these objects which exhibit them \textit{as tracts}, which are equivalent to their standard definitions originally due to M. Krasner in the hyperfield case and Semple \& Whittle \cite{Semple-Whittle96} in the partial field case.

\begin{df}
Let $K$ be a field and let $H \leq K^\times$ be a multiplicative subgroup. Then the quotient monoid $F = K/H := (K^\times / H) \cup \{ 0 \}$ is naturally a tract: the null set $N_F$ consists of all expressions $\sum_{i=1}^k x_i$ such that there exist $c_i \in H$ with
$\sum_{i=1}^k c_i x_i = 0$ in $K$. We call a tract of this form a {\bf quotient hyperfield}. Note that the natural map $\varphi : K \to F$ is a homomorphism of tracts.
\end{df}

\begin{rem}
    As the name ``quotient hyperfield'' suggests, there are more general objects called \textbf{hyperfields}, which can be thought of a fields where the addition operation is allowed to be set valued. Hyperfields are intermediate in generality between quotient hyperfields and tracts. Examples of hyperfields which are not quotient hyperfields were first given in 
    \cite{Massouros85}. 
    All of our examples will be quotient hyperfields, so we can ignore the distinction for the purposes of this paper. A more complete discussion of hyperfields from the point of view of tracts is given in \cite{Baker-Bowler19}.
\end{rem}

\begin{ex}
\begin{enumerate}
    \item The {\bf Krasner hyperfield} $\K$ can be defined as the quotient $K / K^\times$ for any field $K$ with $|K|>2$. Then the resulting tract has two elements $\set{0,1}$, and the null set consists of all sums with at least two $1$s.
    \item The {\bf sign hyperfield} $\S$ is equal to $\R / \R_{>0}$. The resulting tract has three elements denoted $-1,0$ and $1$. The null set consists of sums with at least two non-zero terms of opposite sign.
    \item The {\bf triangle hyperfield} $\V$ is equal to $\C / \S^1$, where $\S^1$ is the complex unit circle. Since we are identifying elements of $\C$ with the same modulus, the resulting tract has elements corresponding to the non-negative real numbers. The null set consists of the sums which could form the side lengths of a convex polygon in the plane, so for example the three-term sums in $N_{\V}$ are those satisfying the triangle inequality. 
    \item The {\bf phase hyperfield} $\P$ is equal to $\C / \R_{>0}$. Here we are identifying elements of $\C$ according to their phase, so the elements of this tract can be identified with $S^1 \cup \set{0}$, although we often choose representatives for non-zero elements which do not lie in $S^1$ to make computation easier. The null set of $\P$ consists of $0$ along with all formal sums of non-zero complex numbers whose convex hull contains the origin. 
    Note that a sum in $\N[\P^\times]$ is null precisely when the vector whose entries are the terms of the sum is colopsided in the sense above.
\end{enumerate}
\end{ex}

Partial fields are another class of algebraic objects which can naturally be viewed as tracts. 

\begin{df} \label{df:partial field}
Let $R$ be a commutative ring with $1$ having unit group $R^\times$, and let $H \leq R^\times$ be a multiplicative subgroup containing $-1$. Then the multiplicative monoid $P = H \cup \{ 0 \}$ is naturally a tract: the null set $N_P$ consists of all expressions $\sum_{i=1}^k x_i$ such that $\sum_{i=1}^k x_i = 0$ in $R$.
We call a tract of this form a {\bf partial field}.\footnote{Note that partial fields are defined differently in \cite{Baker-Lorscheid18}; there the null set is by definition generated by expressions of length at most 3. For our purposes it is simpler to use the present definition.}
\end{df}

\begin{ex} \label{ex:regular partial field}
If we take $R=\Z$ and $H=\{ \pm 1 \}$ in Definition~\ref{df:partial field}, we obtain a tract called the {\bf regular partial field}.
\end{ex}

\subsection{Matroids over tracts}

Let $F$ be a tract and let $E$ be a finite set. For simplicity of notation we identify $E$ with the set $[n] := \{ 1,\ldots, n \}$. 
For $V \in F^n$, the {\bf support} of $V$ is defined to be the set of $i \in [n]$ such that $V_i \neq 0$.

\medskip

For our purposes, it is most convenient to define (strong) $F$-matroids on $E$ as follows. (We assume the reader is familiar with the basic concepts of matroid theory.)

\begin{df}\cite{Baker-Bowler19}\label{def: F matroid}
An {\bf $F$-matroid} $M$ of rank $r$ on $E$ is a matroid $\underline{M}$ of rank $r$ on $E$, together with subsets $\cC(M) \subseteq F^n$ and $\cC^*(M) \subseteq F^n$ (called the {\bf $F$-circuits} and {\bf $F$-cocircuits} of $M$, respectively), such that:
\begin{enumerate}
    \item $\cC(M)$ and $\cC^*(M)$ are both closed under multiplication by elements of $F^\times$.
    \item For any $C \in \cC(M)$, the support of $C$ is a circuit of $\underline{M}$, and for any $C^* \in \cC^*(M)$, the support of $C^*$ is a cocircuit of $\underline{M}$.
    \item For any circuit $\underline{C}$ of $\underline{M}$, there is a projectively unique (meaning unique up to multiplication by some element of $F^\times$) $F$-circuit $C$ whose support is $\underline{C}$, and for any
    cocircuit $\underline{C}^*$ of $\underline{M}$, there is a projectively unique $F$-cocircuit $C^*$ whose support is $\underline{C}^*$.
    \item For any $F$-circuit $C \in \cC(M)$ and any $F$-cocircuit $C^* \in \cC^*(M)$, we have $C \perp C^*$.
\end{enumerate}
We denote the rank of $M$ by $r(M) = r(\underline{M})$.
\end{df}
Note that any matroid has finitely many circuits and cocircuits, so $(3)$ implies that an $F$-matroid has, up to projective equivalence, finitely many $F$-circuits and $F$-cocircuits. As in the case of ordinary matroids, it is enough to specify the $F$-circuits or $F$-cocircuits as they determine each other, although we will often give both in examples to make computation easier.

Hence one may think of an $F$-matroid $M$ as an algebraic object sitting ``above'' the purely combinatorial matroid $\underline{M}$. To aid in readability, we use underlined symbols to denote underlying combinatorial objects and symbols without underlines for their algebraic avatars.

Given a tract $F$, it may or may not be true that a given matroid $\underline{M}$ can be endowed with the structure of an $F$-matroid. In some cases, the class of matroids which may be given an $F$-matroid structure have a name in the literature, including the following examples:

\begin{ex}\label{ex:FMatExs}
\begin{enumerate}
\item {\cite[Example 3.30]{Baker-Bowler19},\:\cite[Proposition 2.19]{Anderson19}} If $F=K$ is a field, a $K$-matroid of rank $r$ on $[n]$ is the same thing as an $r$-dimensional $K$-linear subspace $V$ of $K^n$. To obtain a subspace $V$ from a $K$-matroid $M$, we form a matrix $A$ whose rows are $K$-cocircuits of $M$, one for each projective equivalence class, and let $V$ be the row space of $A$. To recover the $K$-cocircuits of $M$ from $V$, one considers all non-zero vectors in $V$ which have minimal support, and then defines the $K$-circuits of $M$ to be all support-minimal vectors orthogonal to every $K$-cocircuit.

The collection of matroids which may be endowed with the structure of a $K$-matroid is precisely the collection of $K$-representable matroids in the usual sense of matroid theory. 
In this sense, one can view the theory of matroids over tracts as a generalization of the study of subspaces of a vector space.
\item {\cite[Example 3.31]{Baker-Bowler19}} If $F=\K$ is the Krasner hyperfield, a $\K$-matroid is the same thing as a matroid in the usual sense.
\item {\cite[Example 3.33]{Baker-Bowler19}} If $F=\S$ is the sign hyperfield, an $\S$-matroid is the same thing as an {\bf oriented matroid}. 
\item {\cite[Example 3.32]{Baker-Bowler19}} If $F=\T$ is the tropical hyperfield, a $\T$-matroid is the same thing as a {\bf valuated matroid} in the sense of Dress and Wenzel. 

\end{enumerate}
\end{ex}

\begin{rem}\label{rem:F matroid cryptomorphisms}
    As is to be expected in matroid theory, there are many cryptomorphically equivalent ways to axiomatize $F$-matroids. The circuit/cocircuit axioms are most convenient for our purposes, but other options are available and the various equivalences are demonstrated in \cite{Baker-Bowler19}.
\end{rem}

\begin{df}
If $M$ is an $F$-matroid, the {\bf dual $F$-matroid} $M^*$ is the $F$-matroid obtained by replacing $\underline{M}$ with its dual matroid $\underline{M}^*$ and interchanging $F$-circuits and $F$-cocircuits.
\end{df} 

\begin{ex}
    In the case where $F = K$ is a field, $K$-matroid duality has a simple interpretation: it corresponds to interchanging the roles of $V$ and $V^\perp$ (cf.~\Cref{ex:FMatExs}(1)).
\end{ex}

\begin{df}\cite[p.~23]{Baker-Bowler19}
If $M$ is an $F$-matroid, the set $\CapVec(M)$ of {\bf $F$-vectors} of $M$ is the set of all $X \in F^n$ such that $X \perp C^*$ for every $F$-cocircuit $C^*$ of $M$.
Similarly, the set $\Cov(M)$ of {\bf $F$-covectors} of $M$ is the set of all $X \in F^n$ such that $X \perp C$ for every $F$-circuit $C$ of $M$.
\end{df}

\begin{ex}\label{ex:VecCov Examples}
Continuing our running examples:
    \begin{enumerate}
        \item {\cite[Proposition 2.19]{Anderson19}} When $F = K$ is a field, so that our $K$-matroid $M$ is identified with a subspace $V$, we have $\CapVec(M) = V^\perp$ and $\Cov(M) = V$. Thus if $A$ is a matrix representing $\underline{M}$, with null space $\mathrm{Null}(A)$ and row space $\mathrm{Row}(A)$, we have $\mathrm{Null}(A) = \CapVec(M)$ and $\mathrm{Row}(A) = \Cov(M)$.
        \item {\cite[Proposition 5.2]{Anderson19}} If $F=\K$ is the Krasner hyperfield, so that our $\K$-matroid $M$ is just a matroid in the usual sense, then $\CapVec(M)$ can be identified with unions of circuits of $M$ and $\Cov(M)$ with unions of cocircuits.
        \item {\cite[Definition 3.7.1]{Bjorner-LasVergnas-Sturmfels-White-Ziegler99}} If $F=\S$ is the sign hyperfield, then the $\S$-vectors (resp. covectors) are given by taking \textbf{conformal compositions} of $\S$-circuits $C_1,\dots,C_k$ (resp. covectors), which is defined coordinatewise on elements of $\S^{n}$ by
        \begin{align*}
            a\circ b = \begin{cases}
                a & \textrm{if } a \neq 0\\
                b & \textrm{if } a = 0.
            \end{cases}
        \end{align*}
        \item {\cite[Theorem 22]{Bowler-Pendavingh19}} If $F=\T$ is the tropical hyperfield, the $\T$-vectors (resp. $\T$-covectors) are given by taking the element-wise maximum of a finite collection of $\T$-circuits (resp. $\T$-cocircuits). 
\end{enumerate}

Thus one can profitably think of $F$-vectors (resp. $F$-covectors) of an $F$-matroid as a sort of ``$F$-linear closure'' of the $F$-circuits (resp. $F$-cocircuits). In the examples considered above, this can be taken somewhat literally in that the vectors are given by closing the set of circuits under and appropriate operation, but in general this need not be the case.
\end{ex}

\begin{df}\cite[Lemma 3.39]{Baker-Bowler19} \label{df:pushforward}
If $\varphi : F' \to F$ is a homomorphism of tracts and $M'$ is an $F'$-matroid, the {\bf push-forward} $\varphi_*(M')$ is the $F$-matroid on the same underlying matroid whose set of $F$-circuits (resp. $F$-cocircuits) is given by all subsets of $F^n$ of the form $c\varphi(X)$, where $X$ is an $F'$-circuit (resp. $F'$-cocircuit) of $M'$ and $c \in F^\times$. Note that this operation commutes with duality, hence $\varphi_*(M)^* = \varphi_*(M^*)$.
\end{df}


\begin{ex}
As discussed in \Cref{ex:FMatExs} (1), if $K$ is a field then $K$-matroids correspond to subspaces of $K^n$ by taking $V = \mathrm{Null}(A)$ for any matrix $A$ whose rows give representatives for the projective classes in $\cC(M)$. Hence if $\varphi : K \to \K$ is the canonical map from a field $K$ to the Krasner hyperfield, the circuits of the push-forward $\varphi_*(M)$ coincide exactly with the minimally supported non-zero elements of the null space of $A$. More generally, we could replace $A$ with any matrix having the same row space as $A$.
In other words, the collection of $K$-representable matroids, the matroids which appear as pushforwards of $K$-matroids with respect to $\varphi$, and the collection of matroids which admit the structure of a $K$-matroid are all exactly the same.
\end{ex}

\begin{rem}
\begin{enumerate}
    \item The fact that every tract $F$ admits a unique homomorphism to $\K$ thus implies that every $F$-matroid has a unique underlying matroid. 
    \item The fact that there is no morphism from $\K \to \S$ corresponds to the fact that not every matroid is orientable. 
    \item The fact that there is a natural embedding of $\K$ into $\T$ tells us that every matroid can be viewed (in a trivial way) as a valuated matroid.
\end{enumerate}
\end{rem}

\subsection{Reasoning About Vectors and Covectors}
Before moving on to the notions of rank that are the main focus of this paper, we remind the reader of some computational tools and hueristics for thinking about the vectors and covectors of an $F$-matroid.

Note that by \Cref{def: F matroid} (2), the $F$-circuits and $F$-cocircuits of an $F$-matroid $M$ are pairwise orthogonal, so we get inclusions $\cC(M) \subseteq \CapVec(M)$ and $\cC^*(M)\subseteq \Cov(M)$. In general, these inclusions are proper.

In \cite{Anderson19}, the author shows that the $F$-vectors and $F$-covectors of an $F$-matroid determine the $F$-matroid: 

\begin{thm}\cite[Theorem 2.18]{Anderson19}\label{thm:Circuits from Covectors}
    For any $F$-matroid $M$:
    \begin{align*}
        \cC(M) = \Minsupp\set{\CapVec(M) \setminus\set{0}} = \Minsupp\set{\Cov(M)^{\perp}\setminus\set{0}} \\
    \cC^*(M) = \Minsupp\set{\Cov(M)\setminus\set{0}} = \Minsupp\set{\CapVec(M)^{\perp}\setminus\set{0}},
    \end{align*}
    where by $\Minsupp(X)$ we mean the set of vectors in $X$ with minimal support.
\end{thm}

In particular, the $F$-circuits of $M$ are $F$-vectors and the $F$-cocircuits are $F$-vectors.

Mirroring our description of $F$-matroids in terms of their circuits and cocircuits, one cannot hope that \textit{any} collection of vectors forms the set of vectors or covectors for some $F$-matroid \cite[Proposition 2.10]{Anderson19}. In the special case of a field, it is enough to require that the vectors form a subspace and the covectors its orthogonal complement \cite[Proposition 2.19]{Anderson19}, and we will often specify a $K$-matroid by giving a matrix whose null space (resp. row space) gives the set of vectors (resp covectors). 
In the general case, one can take a collection of vectors $S \subseteq F^n$ (thought of as the rows of a matrix) and ask which $F$-matroids contain $S$ among their covectors. This will be a more intricate question to unravel, and form the basis for some of our notions of rank developed in the sequel.

\section{Several notions of rank for matrices over tracts}

Let $F$ be a tract.
In this section we define several different notions of rank for an $m \times n$ matrix $A$ with entries in $F$, and establish some inequalities between them.  As noted in our myriad examples, many of these definitions, specialized to one tract or another, have been of significant prior research interest.

\begin{df}
The {\bf column rank} of $A$, denoted $r_{\col}(A)$, is the maximum number of linearly independent columns of $A$. 
\end{df}

\begin{df}
The {\bf matroidal rank} of $A$, denoted $r_{\mat}(A)$, is the minimum rank of an $F$-matroid $M$ on $[n]$ such that every row of $A$ is a covector of $M$. 
\end{df}

Note that the {\bf Boolean matroid} $U_{n,n}$ on $[n]$ with $[n]$ as its unique basis can be given an $F$-matroid structure $M$ for any tract $F$, with $\cC(M) = \emptyset$ and $\cC^*(M)$ being the standard basis vectors and their multiples. We have $\Cov(M)=F^n$, and thus $r_{\mat}(A) \leq n$ for every $m \times n$ matrix $A$ with entries in $F$. In particular, the matroidal rank of $A$ is well-defined.

\begin{rem}
    If $F=K$ is a field and $A$ represents $\underline{M}$ over $K$, then necessarily every row of $A$ must be a covector of $M$. However the converse need not be the case; for example any matrix with $n$ columns has the property that every row is a $K$-covector of $U_{n,n}$, but not every such matrix represents $U_{n,n}$.
\end{rem}

\begin{rem}
One can also define the {\bf row rank} of $A$, denoted $r_{{\rm row}} (A)$, as the column rank of $A^T$, i.e., $r_{{\rm row}}(A) = r_{\col}(A^T)$.
Similarly, one can define the {\bf transpose matroidal rank} $r_{\tmat}(A) = r_{\mat}(A^T)$. In general, the row and column ranks of a matrix are not equal, nor are the matroidal and transpose matroidal ranks, cf. \Cref{rem:row rank versus column rank} and \Cref{rem:matroidal rank of transpose}.
\end{rem}

The basic general inequality we get is the following.

\begin{prop} \label{prop:col<mat}
$r_{\col}(A) \leq r_{\mat}(A)$, i.e., the column rank of $A$ is at most the matroidal rank of $A$. 
\end{prop}

\begin{proof}
Let $r$ denote the matroidal rank of $A$. It suffices to show that any $r+1$ columns of $A$ are linearly dependent.  Let $\underline{D}$ be any subset of the set $E = [n]$ of columns of $A$ of size $r+1$, and let $M$ be an $F$-matroid realizing the matroidal rank of $A$ (i.e. $M$ is rank $r$ and every row of $A$ is a covector of $M$). Since $\underline{M}$ has rank $r$ and $D$ has $r+1$ elements, we know that $\underline{D}$ must contain a circuit $\underline{C}$ of $\underline{M}$. By \Cref{def: F matroid} (3), there is a (projectively unique) non-zero $F$-circuit $C \in \cC(M)$ of $M$ with support equal to $\underline{C}\subseteq \underline{D}$. By construction, every row of $A$ is orthogonal to every $F$-circuit of $M$, so in particular for any $i \in [m]$,
\begin{align*}
    \sum_{j \in \underline{D}}C_jA_{i,j} = \sum_{j \in [n]} C_jA_{i,j}  \in N_F.
\end{align*}
Thus if we write $A_i$ for the $i$-th column of $A$ we have shown that the set $\set{A_i}_{i\in D}$ is dependent, proving the proposition.
\end{proof}

When $F$ is a field, these various notions of rank all agree and coincide with the ``usual'' notion of rank:
 
\begin{prop}\label{prop:FieldEquality}
If $F = K$ is a field and $A$ is a matrix with entries in $K$, then
\[
r_{{\rm row}}(A) = r_{\col}(A) = r_{\mat}(A) = r_{\tmat}(A).
\]
\end{prop}

The moral of \Cref{prop:col<mat} is that if the rows of $A$ are covectors for an $F$-matroid, then this is witnessed by an $F$-dependence among the columns of $A$. The converse is not true, as the next example shows (i.e., it is possible to have strict inequality in Proposition~\ref{prop:col<mat}):

\begin{ex}\label{exS}
Let $A$ be the following $3 \times 4$ matrix with entries in the sign hyperfield $\mathbb S$, which has $r_{\col}(A) = 2$:

\[
\begin{pmatrix}
1 & -1 & 1 & 1 \\
1 & 1 & -1 & 1 \\
1 & 1 & 1 & -1 
\end{pmatrix}
\]

Suppose there exists a rank 2 oriented matroid $M$ such that the rows of $A$ are covectors of $M$. If $\underline{M}$ is not $U_{2,4}$, then it would have a circuit of rank at most $2$, which would imply that there is an $\S$-linear dependence among at most two columns of $A$, but this is clearly not the case. Hence $\underline{M} = U_{2,4}$, and we will find a contradiction by analyzing the $\S$-circuits.

Since there is a unique $\S$-linear combination of any three columns giving a linear dependence, we compute that $\cC(M)=\{(0,1,1,1),(-1,0,1,1),(-1,1,0,1),(-1,1,1,0)\}$. Next we claim that any three $\S-$circuits of $M$ have to be $\S$-linearly dependent. This follows from \cite[Theorem 2.16]{BakerLorscheidFoundationsI}, with $\underline{M} = U_{2,4}$ and $P = \S$. Because this translation is slightly non-obvious and these objects may be unfamiliar to some readers, we give a careful accounting of this translation; the reader familiar with this material can skip to the last paragraph of this example. 

To begin with, \cite[Theorem 2.16]{BakerLorscheidFoundationsI} is written in terms of \textbf{pastures}, a subclass of tracts containing $\S$. Moving their result closer to the language of this paper and example, they exhibit a bijection:
\begin{align*}
    \set{\S-\text{representations of }U_{2,4}} \xrightarrow{\Xi} \set{\text{Modular systems of } \S-\text{hyperplanes for }U_{2,4}}.
\end{align*}
On the left-hand side, an $F$-representation of $\underline{M}$ is the same thing as a $F$-matroid $M$ with underlying matroid $\underline{M}$\footnote{Actually \cite{BakerLorscheidFoundationsI} is written with a different axiomitization using hyperplanes, but these are equivalent, see \Cref{rem:F matroid cryptomorphisms}.}, so for us the left-hand side is the collection of $\S$-matroids with underlying matroid $U_{2,4}$. 

On the right-hand side, writing $\underline{\cH}$ for the set of hyperplanes of $\underline{M}$, we say that a triple of hyperplanes $(\underline{H}_1,\underline{H}_2,\underline{H}_3)\in \underline{\cH}^3$ is \textbf{modular} if $\underline{G}:=\underline{H}_1\cap \underline{H}_2\cap \underline{H}_3$ is a flat of corank $2$ and $\underline{H}_i\cap \underline{H}_j = \underline{G}$ for all $i\neq j \in \set{1,2,3}$. The hyperplanes of $U_{2,4}$ are precisely the singleton sets $\set{i}, i \in [4]$ so any triple of distinct hyperplanes is modular. A \textbf{modular system of $F$-hyperplanes} for $\underline{M}$ is a collection of vectors\footnote{Written as functions $f_H:E\to F$ in \cite{BakerLorscheidFoundationsI}} $\set{H}_{\underline{H}\in \underline{\cH}}\subseteq F^n$ such that $\supp(H) = \underline{H}$ and if $(\underline{H}_1,\underline{H}_2,\underline{H}_3) \in \underline{\cH}^3$ are modular, then $H_1,H_2,H_3$ are $F$-linearly dependent. Thus in our example, a modular system of $\S$-hyperplanes for $U_{2,4}$ is a collection of vectors $H_1,\dots,H_4 \in \S^4$ such that $\supp(H_i) = \set{i}$ and any three distinct vectors are $\S$-linearly dependent. 

Finally we turn to the map $\Xi$ taking $M$ to a modular system of $\S$-hyperplanes for $\underline{M}$. For each hyperplane $\underline{H} \in \underline{\cH}$, the complement is a cocircuit $\underline{C}^* \in \cC^*(M)$ and the map $\Xi$ associates to $\underline{H}$ the projectively unique $F$-cocircuit with support $\underline{C}^*$ and \cite{BakerLorscheidFoundationsI} asserts that this assignment yields a modular system. What's special in this case is that for $U_{2,4}$, hyperplane complements, i.e. cocircuits, are exactly the same as circuits. Thus for any $\S$-matroid $M$ with underlying matroid $U_{2,4}$, $\Xi$ sends $M$ to a (projectively unique) collection of \textit{$\S$-circuits} $\Xi(M)$ containing a representative supporting every circuit of $\underline{M}$. In other words, up to projective equivalence, $\Xi(M)$ is exactly $\cC(M)$. Since every triple of distinct hyperplanes is modular, every subset of $\Xi(M)$ of size $3$ must be $\S$-linearly dependent, i.e. any three $\S$-circuits of $M$ are $\S$-linearly depdenent.

However, it is easy to check that $\{(0,1,1,1),(-1,0,1,1),(-1,1,0,1)\}$ are not $\S$-linearly dependent, thus there can be no $\S$-matroid $M$ of rank $2$ such that every row of $A$ is a covector of $M$, i.e., $r_{\mat}(A)\geq 3$. One can check that the alternating oriented matroid $C^{4,3}$ (see \Cref{def:alt ori mat}) is a rank $3$ matroid which has every row of $A$ as a covector, so in fact $r_{\mat}(A) = 3$.

\medskip


\end{ex}

\begin{rem} \label{rem:row rank versus column rank}
Example~\ref{exS} also shows that in general $r_{\col}(A) \neq r_{{\rm row}} (A)$. Indeed, the rows of $A$ are $\S$-linearly independent, so $r_{\col}(A^T)=r_{\mat}(A^T)=3$. 

For another example, let $P$ be the regular partial field (cf.~Example~\ref{ex:regular partial field})
and let $A$ be the following matrix with entries in $P$:
\[
\begin{pmatrix}
 1 & -1 & -1 & -1 \\
 1 &  0 &  1 & -1 \\
 1 &  1 &  1 &  1 
\end{pmatrix}
\]
Then one can check that the columns are linearly independent over $P$, hence by \Cref{prop:col<mat} $r_{\col}(A) = r_{\mat}(A)=4$. On the other hand, the rows are linearly independent over $P$ so by the same logic $r_{{\rm row}}(A) = r_{\mat}(A^T)=3$. 
\end{rem}

The following example shows that in general we do {\bf not} have $r_{{\rm row}} (A) \leq r_{\mat}(A)$:

\begin{ex} \label{ex:row and matroidal ranks not equal} 
Let $A$ be the following $3 \times 4$ matrix with entries in the phase hyperfield $\P$:
\[
\begin{pmatrix}
 1 & 1+i & 1 & 0 \\
 2+i &  1+4i &  1 & 1 \\
 2+i &  1+5i &  1 &  1 
\end{pmatrix}
\]
In \cite[Example 4.7, discussion on page 19]{Anderson19}, Anderson constructs a $\P$-matroid $M$ which has the rows of $M$ as covectors. Using \Cref{thm:Circuits from Covectors}, we can recover the $\P$-circuits:
\begin{align*}
    \cC(M) = \set{(-1,1-i,1,0),(1-i,i,0,1),(1-i,0,i-1,-1),(0,i,1-i,1)}.
\end{align*}
From this description it is clear that the rows of $A$ indeed are $\P$-covectors and that the underlying matroid is $U_{2,4}$, showing $r_{\mat}(A) \leq 2$. Conversely, since the rows of $A$ are $\P$-linearly independent, $r_{\col}(A^T) = r_{{\rm row}}(A)=3$. 
\end{ex}

\begin{rem} \label{rem:matroidal rank of transpose}
Examples \ref{rem:row rank versus column rank} and \ref{ex:row and matroidal ranks not equal} also show that in general $r_{\mat}(A) \neq r_{\tmat}(A)$. 
\end{rem}

We now give the promised proof of Proposition~\ref{prop:FieldEquality}.

\begin{proof}[Proof of Proposition~\ref{prop:FieldEquality}]
It is well-known that $r_{{\rm row}}(A) = r_{\col}(A)$, and $r_{\col}(A)\leq r_{\mat}(A)$ due to \Cref{prop:col<mat}. Hence, it is sufficient to show $r_{\mat}(A) \leq r_{{\rm row}}(A)$, or in other words to exhibit a $K$-matroid of rank $r_{{\rm row}}(A)$ such that every row of $A$ is a $K$-covector. If we let $V$ denote the row space of $A$, then by \Cref{ex:VecCov Examples} (1) we obtain a $K$-matroid $M$ of rank $r_{{\rm row}}(A)$ which has as $V$ as its set of $K$-covectors. The row span of $A$ contains the rows of $A$, 
hence $r_{\mat}(A) \leq r_{{\rm row}}(A)$.



\end{proof}

\begin{rem} \label{rem:matroidalrankequalsDeaettrank}
When $F=\K$ is the Krasner hyperfield, the matroidal rank of a matrix $A$ over $\K$ coincides with a notion of rank introduced by Deaett \cite{Deaett}.
Indeed, given an $m \times n$ zero-non-zero matrix pattern $A$, Deaett defines $R(A)$ to be the the collection of all matroids $M$ on ground set $\{1,\ldots,n \}$ such that for each row $R$ of $A$, the set of zero positions of $R$ is a flat of $M$.
Deaett then defines ${\rm mr} R(A)$ to be the minimum rank of a matroid in $R(A)$. 

To see that this coincides with our notion of matroidal rank,
first recall that a $\K$-cocircuit is an ordinary matroid-theoretic cocircuit, 
and the same applies to the other objects in this setting. A flat is an intersection of hyperplanes, a hyperplane is the complement of a cocircuit, and a covector is a union of cocircuits, so the complement of a flat of $M$ is the same thing as a covector of $M$. From this, it follows easily that $r_{\mat}(A) = {\rm mr} R(A)$.  
\end{rem}

\begin{rem} \label{rem:Deaett}
We have already seen in Remark~\ref{rem:matroidal rank of transpose} that in general one does not have $r_{\mat}(A)=r_{\mat}(A^T)$. Example 25 from \cite{Deaett}, when combined with
Remark~\ref{rem:matroidalrankequalsDeaettrank}, provides yet another example. Indeed, according to  \cite[Example 25]{Deaett}, the $8 \times 7$ matrix
\[
\cX = 
\begin{pmatrix}
1 & 0 & 1 & 0 & 1 & 0 & 1 \\
1 & 0 & 0 & 1 & 0 & 1 & 1 \\
1 & 1 & 0 & 0 & 1 & 1 & 0 \\
1 & 1 & 1 & 1 & 0 & 0 & 0 \\
1 & 1 & 1 & 1 & 0 & 0 & 1 \\
0 & 1 & 0 & 0 & 1 & 1 & 0 \\
0 & 0 & 1 & 0 & 1 & 0 & 1 \\
0 & 0 & 0 & 1 & 0 & 1 & 1 \\
\end{pmatrix}
\]
over $\K$ has $ 4 = r_{\mat}(\cX^T) < r_{\mat}(\cX)$.

Given a zero-non-zero pattern $A$, Deaett defines $r_{{\rm tri}}(A)$ to be the maximum $r$ such that we can permute the rows and columns of $A$ in order to obtain a matrix with an $r\times r$ upper-triangular submatrix. 
In particular, $r_{{\rm tri}}(A)=r_{{\rm tri}}(A^T)$ and $r_{{\rm tri}}(A) \leq \min \{ r_{\col}(A), r_{{\rm row}}(A) \}$.
For the matrix above, Deaett shows that $r_{{\rm tri}}(\cX) =4$, so both $r_{\col}(\cX)$ and $r_{{\rm row}}(\cX)$ are at least 4.
Moreover, $r_{\mat}(\cX^T)=4$, so $r_{{\rm row}}(\cX)=4$ and hence $r_{\col}(\cX) = 4$ as well.
This provides another example where we have strict inequality in Proposition~\ref{prop:col<mat}.
\end{rem}

\section{Relative notions of rank}

Suppose $\varphi : F' \to F$ is a homomorphism of tracts. Given a matrix $A$ over $F$, we will define some additional notions of rank which depend on the map $\varphi$ and not just on $F$.

\begin{df}\label{df:pushforwardrank}
The {\bf $\varphi$-matroidal rank} of $A$, denoted $r_{\varphi\text{-}{\mat}}(A)$, is the minimum rank of an $F'$-matroid $M'$ such that every row of $A$ is a covector of $\varphi_*(M')$. 
\end{df}


\begin{df} \label{df:minrank}
We say that a matrix $A'$ over $F'$ is a {\bf lift} of $A$ relative to $\varphi$ if $\varphi(A')=A$. We define $r_{\mat}(\varphi^{-1}(A))$ to be the minimum matroidal rank of a lift of $A$ (or $+\infty$ if $A$ does not lift):
\begin{align*}
    r_{\mat}(\varphi^{-1}(A)) = \min \set{r_{\mat}(A'):\varphi(A') = A}.
\end{align*}
\end{df}

\begin{rem}
By replacing $r_{\mat}$ with a different notion of rank for matrices over $F'$ (e.g. $r_{\col}$), we get other relative rank functions. If $F'$ is a field, we sometimes write $r(\varphi^{-1}(A))$ instead of $r_{\mat}(\varphi^{-1}(A))$ since all of the basic rank functions agree.
\end{rem}

The basic inequality which makes these notions of interest is:

\begin{prop} \label{prop:basic relative inequality}
$r_{\mat}(\varphi^{-1}(A)) \geq r_{\varphi\text{-}\mat}(A) \geq r_{\mat} (A)$.
\end{prop}

\begin{proof}

For the first inequality, let $A'$ be any lift of $A$ achieving the minimum in \Cref{df:minrank} (if no lift exists, there is nothing to prove). By definition, there exists an $F'$-matroid $M'$ of rank $r_{\mat}(\varphi^{-1}(A))$ such that every row of $A'$ is a covector of $M'$. 
By \Cref{df:pushforward}, every row of $A$ is a covector of $\varphi_*(M')$. We therefore have $r_{\mat}(\varphi^{-1}(A)) = r(M') \geq r_{\varphi\text{-}\mat}(A)$. 

For the second inequality, by \Cref{df:pushforwardrank} there exists an $F'$-matroid $M'$ such that rank($M')=r_{\varphi\text{-mat}}(A)$ and every row of $A$ is a covector of $\varphi_*(M')$. The push-forward $\varphi_*(M')$ is an $F$-matroid of rank $r_{\varphi\text{-mat}}(A)$ such that every row of $A$ is a covector of $\varphi_*(M')$. Hence, $r_{\mat}(A) \leq r (\varphi_*(M')) = r_{\varphi\text{-mat}}(A)$.
\end{proof}

Morally, \Cref{prop:basic relative inequality} says that $r_{\varphi\text{-}\mat}(A)$ is a lower bound for the rank of any lift $A'$ of $A$, and this bound is better than the one we would obtain by just using the absolute rank $r_{\mat}(A)$ (or, say, the column rank $r_{\col}(A)$ of $A$, which would give an even worse bound).

\medskip
We can use \Cref{prop:basic relative inequality} to demonstrate the failure of $r_{\mat}(A) = r_{\mat}(A^T)$ in the case of the sign hyperfield $\S$.

\begin{ex}
    Consider the following $8 \times 7$ matrix over $\S$:
    \begin{align*}
        A = \begin{pmatrix}
            -1 & 0  & 1  & 0  & 1 & 0 & 1 \\
            -1 & 0  & 0  & 1  & 0 & 1 & -1 \\
            1  & 1  & 0  & 0  & -1& -1& 0 \\
            1  & -1 & -1 & -1 & 0 & 0 & 0 \\
            1  & 1  & 1  & 1  & 0 & 0 & -1 \\
            0  & -1 & 0  & 0  & 1 & 1 & 0 \\
            0  & 0  & -1 & 0  & -1& 0 & -1 \\
            0  & 0  & 0  & -1 & 0 & -1& 1 \\
        \end{pmatrix}.
    \end{align*}
    First we claim that $r_{\mat}(A^T) = 4$. To see this, note that every column of $A$ is a cocircuit (hence covector) of the dual of the $\S$-matroid structure on the Vámos matroid given in \cite[Example 3.10]{Bland-LasVergnas78}, under the permutation $(3 5 8 4 7) \in S_8$, thus $r_{\mat}(A^T)\leq 4$.

    Let $\varphi:\S \to \K$ denote the unique homomorphism to the Krasner hyperfield. Then $\varphi(A)$ is the same matrix $\cX$ as in \Cref{rem:Deaett}, so by the discussion in that remark, $r_{\mat}(\varphi(A^T)) = r_{\mat}(\cX^T) = 4$. Hence by \Cref{prop:basic relative inequality} we get
    \begin{align*}
        4 = r_{\mat}(\cX^T) \leq r_{\mat}(\varphi^{-1}(\cX^T)) \leq r_{\mat}(A^T) \leq 4.
    \end{align*}

    On the other hand, again following \Cref{rem:Deaett} and applying \Cref{prop:basic relative inequality}, we see that
    \begin{align*}
        4< r_{\mat}(\cX) \leq r_{\mat}(\varphi^{-1}(\cX)) \leq r_{\mat}(A).
    \end{align*}
\end{ex}





\begin{rem}\label{rem: ranks for tropical hyperfields}
For the natural homomorphism $\varphi: \C\{\!\{T\}\!\} \to \T$, we always have the equality $r(\varphi^{-1}(A)) = r_{\varphi\text{-}\mat}(A)$, cf.~\Cref{prop:valuation is epic} below. 
In the tropical algebra literature, this quantity is called the {\bf Kapranov rank} of $A$ relative to the ground field $\C$. More precisely, the Kapranov rank of $A$ is defined to be $r_{\varphi\text{-}\mat}(A^T)$ in \cite[Definition 1.2]{Develin2004} (see section 7 of \cite{Develin2004} for details on the translation), and \cite[Theorem 3.3, Theorem 7.3]{Develin2004} yields  $r(\varphi^{-1}(A)) = r_{\varphi\text{-}\mat}(A^T) = r_{\varphi\text{-}\mat}(A)$. 

As in \cite[Definition 3.9]{Develin2004}, one can change the ground field to obtain a different notion of Kapranov rank. In \cite[Definition 5.3.2]{MaclaganSturmfels}, one finds a notion of Kapranov rank that does not depend on the choice of a valued field $K$; it is defined as the minimum Kapranov rank over all such $K$. By \cite[Theorem 5.3.21]{MaclaganSturmfels}, this Kapranov rank is not necessarily equal to $r_{\mat}(A)$. For example, one can choose $A$ to have rows given by the $\T$-cocircuits of $M$, the non-Pappus matroid given the trivial $\T$-matroid structure. Since $\underline{M}$ is not representable over any field, the Kapranov rank will necessarily be higher than the matroidal rank since a larger rank matroid will be required to push forward. 

For a matrix $A$ over $\T$, there are (at least) two other notions of rank in the literature, namely the {\bf Barvinok rank} \cite[Definition 5.3.1]{MaclaganSturmfels} and the {\bf tropical rank} \cite[Definition 5.3.3]{MaclaganSturmfels}. The tropical rank coincides with the {\bf determinantal rank} defined in \Cref{def: det rank} below. There is a well-known inequality \cite[Theorem 5.3.4]{MaclaganSturmfels} which says that ${\rm tropical \ rank} (A)\leq {\rm Kapranov \ rank} (A) \leq {\rm Barvinok \ rank}(A)$.
\end{rem}
In general, both inequalities in Proposition~\ref{prop:basic relative inequality} can be strict, as the next two examples show.
\begin{ex}\label{ex:Fano T matroid example}
Let $\C\{\!\{T\}\!\}$ denote the field of Puiseux series in the variable $T$. Let $\varphi: \C\{\!\{T\}\!\} \to \T$ be the natural homomorphism of tracts described in \Cref{ex:Tropical}. Let $A$ be the following matrix over $\T$:
\begin{align*}
    \begin{pmatrix}
0 & 0 & 0 & 1 & 1 & 1 & 1 \\
0 & 1 & 1 & 0 & 1 & 1 & 0 \\
0 & 1 & 1 & 1 & 0 & 0 & 1 \\
1 & 0 & 1 & 0 & 1 & 0 & 1 \\
1 & 0 & 1 & 1 & 0 & 1 & 0 \\
1 & 1 & 0 & 0 & 0 & 1 & 1 \\
1 & 1 & 0 & 1 & 1 & 0 & 0
    \end{pmatrix}
\end{align*}
Consider the Fano matroid, viewed in the tautological way as a valuated matroid $M$. 
One can explicitly verify that every row of $A$ is a $\T$-covector of $M$ and that the first 3 columns of $A$ are linearly independent. Using \Cref{prop:col<mat}, it follows that $r_{\mat}(A) = 3$, whereas $r_{\varphi\text{-}\mat}(A) > 3$ (since $\C\{\!\{T\}\!\}$ is a field).

\end{ex}

\begin{ex} [cf.~{\cite[Example 30]{Deaett}}]
Consider the natural homomorphism $\varphi: \F_2 \to \K$, and let $A$ be the following matrix over $\K$:
\[
\begin{pmatrix}
1&0&0&0\\
0&1&0&1\\
0&0&1&1\\
0&1&1&1
\end{pmatrix}
\]

Then $r_{\mat}(\varphi^{-1}(A))=4$, because $\varphi^{-1}(A)$ is a singleton. However, $r_{\varphi\text{-}\mat}(A) = 3$, because one can take $M'$ to be the rank 3 $\F_2$-matroid $M$ represented by the following matrix:
\[
\begin{pmatrix}
1&0&0&0\\
0&1&0&1\\
0&0&1&1\\
\end{pmatrix}.
\]
Recall that by \Cref{ex:VecCov Examples} the $\F_2$-covectors of $M$ are given by the row space of this matrix and the $\F_2$-vectors are given by the null space. By \Cref{thm:Circuits from Covectors}, $M'$ has the $\F_2$-circuit $(0,1,1,1)$. Note that $(0,1,1,1)$ is not a covector of $M'$, but it is a covector of $\varphi_*(M')$ (recall that the null set of $\K$ contains all sums with more than one non-zero entry). 
\end{ex}

What goes wrong in the previous example is that the natural map $\Cov(M') \to \Cov(\varphi_* (M'))$ (i.e., the one given by applying $\varphi$ elementwise to the covectors of $M'$) is not surjective. 
We now show that when the map on covectors is surjective, the situation is nicer. For this, it is convenient to introduce the following definition:

\begin{df}
A homomorphism $\varphi : F' \to F$ of tracts is {\bf epic}\footnote{Not to be confused with the category-theoretic usage of the word ``epic'', which simply means an epimorphism.} if the natural map $\CapVec(M') \to \CapVec(\varphi_* (M'))$
is surjective for every $F'$-matroid $M'$. 
(By duality, this holds iff the natural map $\Cov(M') \to \Cov(\varphi_* (M'))$ is surjective for every $F'$-matroid $M'$.)
\end{df}

Our proof of the following result is inspired by \cite[Theorem 28]{Deaett}.

\begin{thm} \label{thm:epic}
Let $K$ be a field and let $F$ be a tract. If $\varphi : K \to F$ is epic then $r_{\mat}(\varphi^{-1}(A)) = r_{\varphi\text{-}\mat}(A)$ for every matrix $A$ with entries in $F$.
\end{thm}

\begin{proof}
By \Cref{prop:basic relative inequality}, it suffices to show that $r_{\mat}(\varphi^{-1}(A)) \leq r_{\varphi\text{-}\mat}(A)$.
Suppose there exists a $K$-matroid $M$ of rank $r$ such that every row of $A$ is a covector of $\varphi_* (M)$. 
We want to prove that there exists a lift $A'$ of $A$ with $r(A') \leq r$.

\medskip

Since $M$ is a $K$-matroid, we can find a matrix representation $B$ over $K$, i.e., $B$ is a rank $r$ matrix over $K$ such that $M[B] = \underline{M}$. Recall that this means that $\Cov(M) = \row(B)$. However, this does not imply that $\varphi(B) = A$. For any row $v$ of $A$, we know that $v \in \Cov(\varphi_*(M))$, and since $\varphi$ is epic we can find $v' \in \Cov(M) = \row(B)$ such that $\varphi(v') = v$. By doing this for each row $v$ of $A$, we can construct a matrix $A'$ such that $\varphi(A') = A$. Since the rows of $A'$ belong to the row space of $B$ by construction, it follows that
\begin{align*}
    r(A') \leq \dim \row(B) = r(\underline{M}) = r.
\end{align*}
\end{proof} 

The question of whether a given homomorphism $\varphi : K \to F$ is epic or not seems subtle.
For example, the natural map $\varphi : \C \to \P$ is {\bf not} epic, as the following example shows:

\begin{ex} \cite[Section 6.2]{AndersonDelucchi12}
Let $M$ be the $\C$-matroid with underlying matroid $U_{2,4}$ and $\C$-covectors given by the row space of the following matrix:
\[
\begin{pmatrix}
1   & 1+i & 1 & 0 \\
1+i & 4i  & 0 & 1 
\end{pmatrix}.
\]
Let $\varphi : \C \to \P$ be the natural map. Then the $\C$-circuits of $M$ are given by the non-zero vectors of minimal support in the null space of the matrix, namely (up to rescaling): 
\begin{align*}
    \cC(M)= \set{(0,1,-1-i,-4i),(1,0,-1,-1-i),(2,-1+i,0,2+2i),(4,-1+i,-2,0)}.
\end{align*}
Then one can check that $\varphi((2+i, 1+4i,1,1))$ is a $\P$-covector of the pushforward $\varphi_*(M)$, but $(2+i, 1+4i,1,1)$ is not in the row span of the above matrix, hence $\varphi((2+i, 1+4i,1,1)$ is not the image of a $\C$-covector of $M$. Thus the induced map from $\Cov(M)$ to $\Cov(\varphi_* (M))$ is not surjective.
\end{ex}

On the positive side, the proof of {\cite[Theorem 28]{Deaett}} immediately gives:

\begin{lemma}
If $K$ is an infinite field, the canonical map $\varphi : K \to \K$ is epic.
\end{lemma}

The main idea of the proof is that if $M'$ is a $K$-matroid with underlying matroid $M$ (i.e. $\varphi_*(M') = M$), the map from $K$-circuits of $W$ to circuits of $M$ is surjective by definition (cf.~Definition~\ref{df:pushforward}). Recall that a vector of a $\K$-matroid is just a union of $\K$-circuits and a $K$-vector of a $K$-matroid is just a $K$-linear combination of $K$-circuits. Hence given any $\K$-vector $C$ of $M$ given as the union of $\K$-circuits $C_1,\dots,C_n$, it suffices to take a linear combination of preimages of each $C_i$ such that there is no cancellation (which is possible since $K$ is infinite). Then clearly this linear combination is mapped to $C$.

\medskip


As another example of a positive result, we can say the following about the sign case.

\begin{prop} \label{prop:sign is epic}
Suppose $K$ is a field such that $\Q \subseteq K \subseteq \R$. The natural map ${\rm sign} : K \to \S$ is epic.
\end{prop}
 
\begin{proof}
Let $M$ be an $K$-matroid on $[n] = \{ 1,\ldots, n \}$ and let ${\rm sign}_* (M)$ be the associated oriented matroid. 
By \Cref{df:pushforward}, the circuits of ${\rm sign}_* (M)$ are the push-forward of the circuits of $M$, i.e. $C' \in \cC({\rm sign}_* (M))$ iff $C' = {\rm sign}(C)$ for some $C \in \cC(M)$.

On the other hand, recall that by the vector axioms for oriented matroids 
{\cite[Definition 3.7.1]{Bjorner-LasVergnas-Sturmfels-White-Ziegler99}}, the vectors of an oriented matroid are precisely the conformal compositions of circuits (recall \Cref{ex:VecCov Examples} part (3)). 
Hence, it suffices to show that the image of the natural map $\Cov(M) \to \Cov({\rm sign}_* (M))$ is closed under conformal composition.

Suppose $X'$ and $Y'$ are two covectors in  $\Cov({\rm sign}_* (M))$ and that there exist $X$ and $Y$ in $\Cov(M)$ satisfying $X' = {\rm sign} (X)$ and $Y' = {\rm sign} (Y)$. 
Let $\epsilon$ be a positive element of $K$ that is smaller than $|\tfrac{X_i}{Y_i}|$ for all $i\in [n]$. Then $X+\epsilon Y$ is a covector of $M$ and ${\rm sign}(X+\epsilon Y)$ gives the conformal composition $X'\circ Y'$.
\end{proof}

A similar argument can be used to show:

\begin{prop} \label{prop:valuation is epic}
Let $K$ be an infinite field. The natural map ${\rm exp}({-v}) : K\{\!\{T\}\!\}  \to \T$ is epic.
\end{prop}

\begin{proof}
As in the proof of Proposition~\ref{prop:sign is epic}, there is a binary composition operation for vectors of valuated matroids, defined by $(X\circ Y)_i = \text{max} \{X_i,Y_i\}$, such that vectors are precisely the finite compositions of circuits, \cite[Theorem 22]{Bowler-Pendavingh19} (recall \Cref{ex:VecCov Examples} part (4)). It suffices to show that the image of the natural map $\Cov(M) \to \Cov(v_* (M))$ is closed under composition.

Suppose $X'$ and $Y'$ are two covectors in  $\Cov(v_* (M))$ and let $X$, $Y$ in $\Cov(M)$ satisfying $X_i' = v (X_i)$ and $Y_i' = v (Y_i)$ for all $i\in [n]$. Let $c_i$ and $d_i$ be the coefficients of the initial terms of $X_i$ and $Y_i$. Since $K$ is an infinite field, there exists some $a \in K$ such that $a\cdot d_i \neq -c_i$ for all $i$. This implies that there is no cancellation in the lowest order term in each coordinate, hence that $v(X+a\cdot Y) = X\circ Y$.
\end{proof}

\begin{ex}\label{ex:SignTrop Epic}
The natural maps $\S\to\K$ and $\T\to\K$ are both epic. This follows from the fact that a vector of a matroid is the same thing as a union of circuits, together with the observation that the composition operations in \Cref{prop:sign is epic} and \Cref{prop:valuation is epic} satisfy $\supp (V_1\circ V_2) = \supp(V_1) \cup \supp(V_2)$. 
\end{ex}

\begin{rem}
    Observing the features of \Cref{prop:sign is epic}, \Cref{prop:valuation is epic}, and \Cref{ex:SignTrop Epic}, one might hope that all surjective maps $\varphi:F' \to F$ from infinite tracts to finite ones are epic, but this is not the case. 

    Consider the tract $F$ with the same underlying set and multiplicative structure as $\R$, but with the null set restricted to formal sums $a + (-a)$ with at most two non-zero terms. Let $M$ denote the $F$-matroid with underlying matroid $U_{2,3}$, $F$-circuit set $\cC(M) = \set{c\cdot (1,1,1): c\in \R},$ and $F$-cocircuit set $\cC^*(M) = \set{c(e_i-e_j):i< j \in \set{1,2,3},c\in \R}.$ Then, if $\varphi : F \to \K$ denotes the canonical map to $\K$, we see that $(1,1,1)$ is covector of $\varphi_*(M)$ which is not the image of  any covector of $M$ (since, for example, a covector of $M$ can have at most two non-zero terms).
    
    This example highlights the importance of the composition operations that enabled the proofs of the previous three results. In \cite[Section 4.3]{Bowler-Pendavingh19} the authors describe, in the case where $F$ is a \textbf{stringent hyperfield}, a general composition operation on the $F$-circuits of $F$-matroids such that $F$-vectors are precisely compositions. This suggests that perhaps if $F$ is a finite stringent hyperfield and $F'$ is infinite, then $\varphi$ will be epic.
\end{rem}

In general, it seems hard to say precisely when equality holds for the various inequalities we've touched upon so far in this paper. 
However, there is at least one case where things are relatively nice.

\begin{thm} \label{thm:quotienthyperfieldrank}
Let $K$ be a field, let $H$ be a subgroup of $K^\times$, and let $\varphi: K \to F$ be the canonical quotient map to the hyperfield $F = K / H$. 
Let $A$ be an $m\times n$ matrix over $F$. 
Then the following are equivalent:
\begin{enumerate}
    \item $r_{\mat}(\varphi^{-1}(A)) = n$. 
    \item $r_{\varphi\text{-}\mat}(A) = n$.
    \item $r_{\mat} (A) = n$.
    \item $r_{\col} (A) = n$.
\end{enumerate}
\end{thm}

\begin{proof}
In view of \Cref{prop:col<mat} and \Cref{prop:basic relative inequality} we have
\begin{align*}
    r_{\mat}(\varphi^{-1}(A)) \geq r_{\varphi\text{-}\mat}(A) \geq r_{\mat} (A) \geq r_{\col} (A).
\end{align*}
Hence it suffices to show $r_{\mat}(\varphi^{-1}(A)) = n$ implies $r_{\col} (A) = n$. We show the inequality that $r_{\col} (A) < n$ implies $r_{\mat}(\varphi^{-1}(A)) < n$.

If $r_{\col}(A)<n$, then the columns of $A$ are linearly dependent over $F$, hence there exist $x_1,\dots, x_n \in F$, not all zero, such that $\sum_{j = 1}^n a_{ij}x_j \in N_F$ for each $i \in \set{1,\dots,m}$. If we let $\widetilde{a_{ij}}$ be any lift of $a_{ij}$ and let $\widetilde{x_j}$ be any lift of $x_j$, then by definition of $F$ there exist $c_{ij}\in H$ such that $\sum_{j = 1}^nc_{ij}\widetilde{a_{ij}}\widetilde{x_j} = 0$ for all $i$. If we take the lift $A'$ of $A$ to be given by $A'_{i,j}:= c_{ij}\widetilde{a_{ij}}$ then clearly $A'$ has columns linearly dependent over $K$, so $r(A')<n$. Since $K$ is a field, \Cref{prop:FieldEquality} implies that $r_{\mat}(A') <n$, so we have exhibited a lift of $A$ which has matroidal rank strictly less than $n$, proving the theorem.
\end{proof}

\begin{rem}
Note that \Cref{thm:quotienthyperfieldrank} would not hold if we were to replace $n$ in the four equivalent statements with some $r<n$. The subtlety here is that in the case where there is a linear dependence among all of the columns of the matrix, we can lift the whole matrix simultaneously to $K$ and find the required coefficients $c_1,\dots,c_n$ corresponding to this lift. By way of contrast, if there is a linear dependence among every $r$ of the columns of $A$ for some $r<n$, we would get linear dependencies among the lifts of each subset of $r$ columns and corresponding coefficients $\set{(c_1,\dots,c_r)_D:D \in {n \choose r}}$ \textit{for each subset}, but these may not agree on the overlaps of the different lifts. 

When $m<n$ the result is vacuously true ($r_{\mat}(\phi^{-1}(A))\leq m<n$ by basic linear algebra) but not very helpful. Indeed, we already saw in Example~\ref{exS} that there exists a $3\times 4$ matrix $A$ over $\S$ with $r_{\mat}(A) = 3$ but $r_{\col} (A) = 2$.
\end{rem}

\begin{cor} \label{cor:quotienthyperfieldrank}
With notation as in Theorem~\ref{thm:quotienthyperfieldrank}, the following are equivalent for an $n\times n$ matrix $A$ over $F$:
\begin{enumerate}
    \item Every matrix $A'$ over $K$ with $\varphi(A') = A$ is nonsingular.
    \item The columns of $A$ are $F$-linearly independent.
    \item $r_{\mat} (A) = n$.
    \item $r_{\tmat} (A) = n$.
\end{enumerate}
\end{cor}

\begin{proof}
    This is a direct consequence of \Cref{thm:quotienthyperfieldrank}, with the only slight wrinkle being the application of third equality in \Cref{prop:FieldEquality} to get part $(4)$, namely:
    \begin{align*}
        r_{\mat}(A) = n \xleftrightarrow{\ref{cor:quotienthyperfieldrank}} r_{\mat}(\varphi^{-1}(A)) = n \xleftrightarrow{\ref{prop:FieldEquality}}  r_{\tmat}(\varphi^{-1}(A)) = n \xleftrightarrow{\ref{cor:quotienthyperfieldrank}} r_{\tmat}(A) = n.
    \end{align*}
\end{proof}

Applying Corollary~\ref{cor:quotienthyperfieldrank} to various specific examples 
recovers several known results from the literature in a unified manner. 

\begin{proof}[Proof of ~\ref{theorem:thm2}]
    For part $(1)$, consider the map $\varphi:\C \to \V$ and take $A$ to be a matrix over $\V$. By \Cref{cor:quotienthyperfieldrank}, every complex matrix $A'$ with $\varphi(A') = A$ is non-singular if and only if the columns of $A$ are $\V$-linearly independent. Hence, after left multiplying by a permutation matrix to put the largest entries on the diagonals, strict diagonal dominance is equivalent the condition that for each row of $PA$, if one takes the associated linear combination with coefficients from $D$, the resulting sum is not in $N_\V$, giving the theorem.

    For part $(2)$, we instead take $\varphi:\C \to \P$ and we now take $A$ to be a matrix over $\P$. By \Cref{cor:quotienthyperfieldrank}, every complex matrix $A'$ with $\varphi(A') = A$ is non-singular if and only the columns of $A$ are $\P$-linearly independent. Since linear dependence over $\P$ is equivalent to the failure of each row to be colopsided after some rescaling, the result follows.
\end{proof}


Applied to ${\rm sign} : \R \to \S$, \Cref{thm:quotienthyperfieldrank} recovers the following result from \cite{arav2015}:

\begin{cor} \cite[Corollary 21]{arav2015}
 Let $A$ be an $m\times n$ sign pattern and let {\rm row}$(A)$ denote the set of rows of $A$. Then $r({\rm sign}^{-1}(A)) = n$
if and only if for every non-zero sign vector $x\in \{+,-,0\}^n$, {\rm row}$(A)\nsubseteq x^\perp$.
\end{cor}

\begin{proof}
    By the equivalence of $(1)$ and $(2)$ from \Cref{thm:quotienthyperfieldrank}, we have $r(\sign^{-1}(A)) = n$ if and only if the columns of $A$ are $\S$-linearly independent. Thus there is no non-zero vector with entries in $\S$ which is orthogonal to every row of $A$, as this would exhibit an $\S$-linear dependence among the columns.
\end{proof}

Here is an application of \Cref{thm:epic} inspired by the proof of \cite[Proposition 2.5]{Berman2008AnUB} given in {\cite[unlabeled Theorem following Corollary 29]{Deaett}}.

\begin{thm}  \label{thm:zeroPatternApplication}
Let $\chi$ be an $m\times n$ matrix over $\K$, i.e., a zero-non-zero pattern. Let $K$ be an infinite field, and let $\varphi: K \to \K$ be the natural map.
If $\chi$ has at least $t$ non-zero entries in each row, then $r(\varphi^{-1}(\chi)) \leq n-t+1$, i.e., there exists a matrix $A$ over $K$ with  $\varphi(A) = \chi$ and $r(A) \leq n-t+1$. 
\end{thm}

\begin{proof}
Any $v\in \K^n$ with at least $t$ non-zero elements is a covector of $U_{n-t+1,n}$, since the circuits are all subsets of size $n-t+2$. 
Furthermore, the matroid $U_{n-t+1,n}$ is representable over any infinite field \cite[Corollary 12.2.17]{Oxley92} (for example, by an $(n-t+1)\times n$ Vandermonde matrix); let $A$ denote such a matrix. Then the row space of $A$ gives a $K$-matroid $M$ satisfying $\varphi_*(M) = U_{n-t+1,n}$. To conclude, we note that by \Cref{thm:epic},
\begin{align*}
    r_{\mat}(\varphi^{-1}(\chi)) = r_{\varphi\text{-}{\rm mat}}(\chi) \leq r(M) = r(A) = n-t+1.
\end{align*}
\end{proof}
\begin{rem}
    Note that this result does not imply that every $r(A)\leq n - t+1$ for every matrix $A$ such that $\varphi(A) = \chi$. For example, take $n = 2$ and $t = 2$ and consider the $K$-matrices
    \begin{align*}
        A_1 = \begin{pmatrix}
            1 & 1\\
            1 & 1
        \end{pmatrix},
        \quad
        A_2 = \begin{pmatrix}
            1 & 2\\
            1& 1
        \end{pmatrix}.
    \end{align*}
    Both $A_1$ and $A_2$ map to the same $\K$-matrix and have $t$ non-zero entries in each row, yet $r(A_1)  =1\leq n-t+1$ while $r(A_2) = 2> n-t+1$. This highlights the importance in the proof of constructing of a $K$-matrix which represents $U_{n-t+1,n}$ (as $A_1$ does), which is stronger than just being in the preimage of $\chi$ (which both $A_1$ and $A_2$ are).
\end{rem}

Proposition~\ref{prop:sign is epic}, together with the idea behind the proof of Theorem~\ref{thm:zeroPatternApplication}, allows us to obtain a new proof of \cite[Theorem 3.3]{arav2015}, a similar result to \Cref{thm:zeroPatternApplication} involving sign patterns. Before stating the theorem, we will need the following definitions.

\begin{df}\cite[p.~903]{LYMFW13}
Given $V \in \{0,1,-1\}^n$, let $V^+ = \{ i \; | \; V_i = 1 \}$ and $V^- = \{ i \; | \; V_i = -1 \}$.
Define the {\bf number of polynomial sign changes} $\psc(V)$ to be the maximal number of sign changes of a sign vector $X\in \{1,-1\}^n$ such that $V^+\subseteq X^+$ and $V^-\subseteq X^-$. (In other words, we allow a zero entry of $V$ to count as either $1$ or $-1$ and then count the maximum possible number of sign changes in such a vector.)
\end{df}

We will also make use of the alternating oriented matroid $C^{n,r}$ of rank $r$ on $[n]$. A more classical definition can be found in \cite[Section 9.4]{Bjorner-LasVergnas-Sturmfels-White-Ziegler99}, but for the reader's convenience we furnish a definition adapted to our notation:

\begin{df}\label{def:alt ori mat}
    The \textbf{alternating oriented matroid} $C^{n,r}$ is the $\S$-matroid with underlying matroid $U_{r,n}$ for which $\cC(C^{n,r})$ consists of all vectors in $\S^n$ with exactly $r+1$ non-zero entries such that the non-zero entries alternate in sign.
\end{df}

\begin{lemma} \label{lemma:CovectorOfCnr}
$V\in \S^n$ is a covector of $C^{n,r}$ if $\sigma(V) < r$. 
\end{lemma}
\begin{proof}
Suppose for the sake of contradiction that $\psc(V)<r$ and $V$ is not an $\S$-covector of $C^{n,r}$. Then $V$ is not orthogonal to some $\S$-circuit $C \in \cC(C^{n,r})$. This implies that $\sum_{i = 1}^n V_i\cdot C_i \not \in N_\S$, so each $V_i\cdot C_i$ has the same sign (or is zero). Since the circuits of an $\S$-matroid are closed under multiplication by $\S^\times$, we may assume that without loss of generality that $V_i\cdot C_i \in \{ 0,1 \}$ for all $i\in [n]$. Define
$$
V'_i=
\begin{cases}
V_i, \text{ if }V_i \neq 0\\
C_i, \text{ if }V_i = 0 \text{ and }C_i \neq 0 \\
1, \text{ if }V_i = C_i = 0.
\end{cases}
$$
Then $V'$ is a vector in $\{X\in \{1,-1\}^n \mid V^+\subseteq X^+, V^-\subseteq X^-\}$, so $\psc(V')\leq \psc(V)<r$. On the other hand, $V'_i\cdot C_i = 1$ for all $i\in {\rm supp}(C)$. Hence, the restriction of $V'$ to the support of $C$ has $r$ sign changes, which yields $\psc(V')\geq r$, a contradiction.
\end{proof}


\begin{thm}\cite[Theorem 3.3]{LYMFW13} \label{thm:signPatternApplication}
Let $\chi$ be an $m\times n$ matrix over $\S$, i.e., a sign pattern. Let ${\rm sign}: \Q \to \S$ be the natural map.
If the number of polynomial sign changes for each row $\chi$ is less than $k$, then $r({\rm sign}^{-1}(\chi)) \leq k$, i.e., there exists matrix $A$ over $\Q$ with sign pattern $\chi$ such that $r(A) \leq k$.
\end{thm}

\begin{proof}


The $\S$-matroid $C^{n,k}$ is realizable over $\Q$ by \cite[Proposition 9.4.1]{Bjorner-LasVergnas-Sturmfels-White-Ziegler99}, so in particular there is a $\Q$-matroid $M$ such that $\sign_*(M) = C^{n,k}$. By Lemma~\ref{lemma:CovectorOfCnr}, every row of $\chi$ is a covector of $C^{n,k}$. By \Cref{thm:epic} and \Cref{prop:sign is epic}, we have
 $r({\rm sign}^{-1}(\chi)) = r_{{\rm sign}\text{-}\mat}(A) \leq r(M) = r(C^{n,k})=k$. Hence by \Cref{df:minrank} there must exist a lift $A$ of $\chi$ with $r(A)\leq k$.
\end{proof}

\begin{rem}
    Philosophically, one can summarize our new proofs of the ``classical'' results in this section as arising from viewing those results as ``shadows'' of general statements about homomorphisms between tracts. This perspective is essentially the same as the approach of the paper 
\cite{BL2020} by the first author and Lorscheid. In that paper, Baker and Lorscheid define a notion of {\bf multiplicity} for roots of a polynomial over a hyperfield $F$, and 
given a homomorphism $\varphi: F' \to F$ of hyperfields, they prove a general inequality relating the multiplicities of roots of a polynomial $p \in F'[x]$ and its image $\varphi(p)$ in $F[x]$ (along with some sufficient conditions for equality to hold). They also show how both Descartes' Rule of Signs and Newton's Polygon Rule are special cases of this general inequality. 
\end{rem}

\section{Some open questions}


\subsection{Determinantal rank}

\begin{df} \label{def: det rank}
The {\bf determinant} of an $n\times n$ matrix $B=\{b_{i,j}\}$ over a tract $F$ is the following formal sum, thought of as an element of $\N[F^\times]$:
\[
\sum_{\sigma\in S_n} (-1)^{{\rm sgn}(\sigma)} b_{1,\sigma(1)}\cdots b_{n,\sigma(n)}.
\]

The {\bf determinantal rank} of $A$ (possibly non-square), denoted $r_{\det} (A)$, is the maximal $r$ such that $A$ has a $r \times r$ submatrix $B$ with $\det(B)\notin N_F$.
\end{df}

One has the following inequality, which follows from a non-trivial theorem of Dress and Wenzel. 

\begin{thm} [{\cite[Theorem 4.9]{Dress-Wenzel92}}]\label{thm:perfect inequality}
If the tract $F$ is perfect\footnote{This means that for every $F$-matroid $M$, every $F$-vector of $M$ is orthogonal to every $F$-covector of $M$. Examples of perfect tracts include fields and the hyperfields $\K,\S,\T$, see \cite{Baker-Bowler19}.}, then $r_{\mat}(A)\geq r_{\det}(A).$ 
\end{thm}

\begin{proof}
    The only translation required is to note that the perfect fuzzy rings studied by Dress and Wenzel are a subcategory of perfect tracts, and {\em a fortiori} matroids with coefficients in a perfect fuzzy ring in the sense of Dress--Wenzel are precisely $F$-matroids in our sense, where $F$ is the corresponding perfect tract \cite[Subsection 2.7]{Baker-Bowler19}. The set $K_0$ in the terminology of Dress--Wenzel corresponds to $N_F$, and their condition on the rows of $A$ is equivalent to requiring that the rows of $A$ be covectors of the given matroid. Hence their result says that any square submatrix of size larger than the rank of the matroid has null determinant, which is what we've claimed.
\end{proof}

\begin{rem} \label{rem:Deaett2}
The inequality $r_{\mat}(A)\geq r_{\det}(A)$ over a perfect tract $F$ can be strict. Indeed, the determinantal rank is clearly invariant under taking transposes, so any matrix with $r_{\mat}(A) \neq r_{\mat}(A^T)$ is an example of the strictness of the inequality, e.g., the matrix $\cX$ from Remark~\ref{rem:Deaett}. 
\end{rem}

\begin{rem}
    To highlight the importance of the ``perfect'' hypothesis in \Cref{thm:perfect inequality}, consider the tract $F$ which has the same underlying set as $\K$, but for which the null set consists only of sums with at most three non-zero terms. Then consider the $3\times 3$ matrix 
    \begin{align*}
        A = \begin{pmatrix}
        1 & 1 & 1\\
        1 & 1 & 1\\
        1 & 1 & 1
    \end{pmatrix}.
    \end{align*}
    Note that we can construct an $F$-matroid $M$ which has a single $F$-circuit $\set{(1,1,1)}$ and $\cC^*(M) = \set{(1,1,0),(1,0,1),(0,1,1)}$, so $\underline{M} = U_{2,3}$. Clearly every row of $A$ is an $F$-covector of $M$, hence $r_{\mat}(A) \leq 2$. On the other hand, if we compute the determinant of $A$ we will get a 6-term sum, and no such sums are null, hence $\det(A) \not \in N_F$. Thus $r_{\det}(A) = 3$.
\end{rem}

\begin{question}
What can one say about the relationship between $r_{\det}(A)$ and $r_{\col}(A)$ over a perfect tract $F$?
\end{question}

In this direction, there are a couple of known results for square matrices when $F=\S, \T$, or $\K$:
\begin{thm}\label{thm:full rank square matrix}
Let $A$ be an $n\times n$ square matrix over $F = \T$ \cite[Theorem 3.6]{izhakian2008}, $\S$\cite[Theorem 1.2.5]{brualdi_shader_1995}, or $\K$. Then $r_{\det}(A) = n$ if and only if  $r_{\col}(A) = n$. 
\end{thm}

\begin{proof}
    In the case where $F = \K$, the conditions that $r_{\col}(A) = n$ and $r_{\det}(A) = n$ are both clearly equivalent to $A$ having exactly one non-zero entry in each row or column.
\end{proof}

Using the $\T$ case of Theorem~\ref{thm:full rank square matrix} as a building block, Izhakian and Rowen prove the following more general result (their submatrix rank is the same as our determinantal rank):
\begin{thm}\cite[Theorem 3.4]{izhakian2008} \label{thm:izhakian-rowen}
Let $A$ be an $m\times n$ matrix over $\T$. Then $r_{\det}(A)=r_{\col}(A)=r_{{\rm row}}(A)$.
\end{thm}

\begin{rem}
In conjunction with \cite[Theorem 5.3.4]{MaclaganSturmfels} (mentioned in \Cref{rem: ranks for tropical hyperfields}), \Cref{prop:basic relative inequality}, \Cref{thm:epic}, and \Cref{thm:perfect inequality} give a chain of inequalities for matrices over $\T$:
\begin{align*}
    r_{\mat}(\varphi^{-1}(A)) = r_{\varphi\text{-}{\rm mat}}(A) \geq r_{\mat}(A) \geq r_{\det}(A) = r_{\col}(A) = r_{{\rm row}}(A).
\end{align*}

The two displayed inequalities can both be sharp, as can be seen from \Cref{ex:Fano T matroid example} in the first case and from Remark~\ref{rem:Deaett2} in the second. 
\end{rem}

\begin{rem}
The analogue of \Cref{thm:izhakian-rowen} does {\bf not} hold over $\S$, as one sees from \Cref{rem:row rank versus column rank}. This highlights how the property of having determinantal rank equal to column rank is rather special to $\T$.
\end{rem}

\subsection{Rank of $A$ versus rank of $A^T$}

We have seen in \Cref{rem:row rank versus column rank}, in which $F = \S$, that the column rank of a matrix $A$ over a tract $F$ is not always equal to the row rank. Of course, the two ranks are equal when $F$ is a field. 

\begin{question}
Can we characterize the tracts for which $r_{\col}(A) = r_{{\rm row}}(A)$ for all matrices $A$ over $F$, or at least give a nontrivial sufficient condition for this to hold?
\end{question}

Similarly, we have seen in Remark~\ref{rem:matroidal rank of transpose} that, unlike the case of fields, the matroidal ranks of $A$ and $A^T$ are not always equal. But we would like to understand the boundary of this failure more generally:

\begin{question}
Can we characterize the tracts $F$ for which $r_{\mat}(A) = r_{\mat}(A^T)$ for all matrices $A$ over $F$, or at least give a nontrivial sufficient condition for this to hold?
\end{question}

\subsection{Matrices versus systems of linear equations}

In linear algebra over a field $K$, solving a system of homogeneous linear equations is equivalent to studying the null space of the matrix $A$ of coefficients, and the rank-nullity theorem applied to $A$ shows that if there are more unknowns than equations then there is a non-zero solution. For a system of homogeneous linear ``equations'' over a tract $F$, each of the form $\sum_j a_{ij} X_j \in N_F$, we can still view the solution set as the null space of a matrix, but the rank-nullity theorem no longer holds in general, and a non-zero solution does not always exist.

\begin{ex}
Consider the following $2 \times 3$ matrix $A = (a_{ij})$ over the regular partial field (\Cref{ex:regular partial field}):
$$
\begin{pmatrix}
1&-1&-1 \\
1&0&1
\end{pmatrix}
$$
For the corresponding system of homogeneous linear ``equations'', there are more unknowns than equations but $(0,0,0)$ is the only solution.  
\end{ex}

\begin{question} \label{question:system of linear equations}
If $F$ is a {\bf hyperfield}, does a system of homogeneous linear equations with more unknowns than equations always have a non-zero solution?
\end{question}

If $F = K/H^\times$ is a quotient hyperfield then the answer to Question~\ref{question:system of linear equations} is yes, since we can deduce the result directly from the corresponding result over $K$. In \cite[Theorem 3.11]{hobbyJunFETVINS2024}\footnote{Written after the first version of this paper was posted to the arXiv.}, the authors construct an infinite family of non-quotient hyperfields for which more unknowns than equations implies a non-zero solution. However, Question~\ref{question:system of linear equations} remains open in general.


\begin{small}
\bibliographystyle{plain}
\bibliography{Matrix_Rank_Revised}
\end{small}

\end{document}